\documentclass[english, a4, 12pt]{amsart}

\usepackage{multicol} 
\usepackage{babel}
\usepackage{amstext}
\usepackage{amsmath}
\usepackage{amsfonts}
\usepackage{latexsym}
\usepackage{ifthen}
\usepackage{xypic}
\usepackage{amssymb}

\xyoption{all}
\newcommand{\Formel}[1]{(\ref{#1})}

\newcommand{\U}{\mathfrak U}

\renewcommand{\L}{\mathbb L}

\newcommand{\BN}{\mathbb B}

\newcommand{\Lem}[1]{Lemma~\ref{#1}}
\newcommand{\Prop}[1]{Proposition~\ref{#1}}
\newcommand{\Cor}[1]{Corollary~\ref{#1}}

\newcommand{\Theo}[1]{Theorem~\ref{#1}}
\newcommand{\F}{{\mathcal F}}
\newcommand{\G}{{\mathcal G}}

\newcommand{\id}{{\rm id}}
\renewcommand{\O}{{\mathcal O}}

\newcommand{\PN}{{\mathbb P}}

\newcommand{\rk}{{\rm rk}}

\newcommand{\Pic}{{\rm Pic}}

\newcommand{\lra}{\longrightarrow}
\newcommand{\KC}{{\mathbb C}}
\newcommand{\KZ}{{\mathbb Z}}

\newcommand{\Hom}{{\rm Hom}}

\newcommand{\HAut}{{\rm HAut}}
\newcommand{\Aff}{{\rm Aff}}
\newcommand{\COAff}{CO{\rm Aff}}

\newcommand{\K}{\mathcal K}
\newcommand{\D}{\mathcal D}

\newcommand{\KR}{\mathbb R}

\newcommand{\OG}{{\mathbb O}{\mathbb G}}
\renewcommand{\G}{{\mathbb G}}
\newcommand{\Q}{{\mathcal Q}}

\newtheorem{lemma1}[equation]{}

\newenvironment{lemma}{\begin{lemma1}{\bf Lemma.}}{\end{lemma1}}

\newenvironment{example}{\begin{lemma1}{\bf Example.}\rm}{\end{lemma1}}

\newenvironment{theorem}{\begin{lemma1}{\bf Theorem.}}{\end{lemma1}}

\newenvironment{proposition}{\begin{lemma1}{\bf Proposition.}}{\end{lemma1}}
\newenvironment{corollary}{\begin{lemma1}{\bf Corollary.}}{\end{lemma1}}

\renewcommand{\L}{\mathcal L}

\begin{document}

\begin{abstract}
We classify projective manifolds with flat holomorphic conformal structures.
\end{abstract}
\title{Projective manifolds modeled after hyperquadrics}
\author[P. Jahnke]{Priska Jahnke}
\address{Priska Jahnke - Fakult\"at T1- Hochschule Heilbronn
  - Max-Planck-Stra{\ss}e 39 - D-74081 Heilbronn, Germany}
\email{priska.jahnke@hs-heilbronn.de}
\author[I.Radloff]{Ivo Radloff}
\address{Ivo Radloff - Mathematisches Institut - Universit\"at T\"ubingen -
  D-95440 T\"ubingen, Germany}
\email{ivo.radloff@uni-tuebingen.de}
\date{\today}
\maketitle

\section*{Introduction}

Manifolds with certain reductions of the structure group are classical objects in real differential geometry. In the holomorphic setting, compact complex manifolds modeled after hermitian symmetric spaces have been studied by Kobayashi and Ochiai (\cite{KO}, \cite{KObook}, \cite{KO}) and others, more recently by Mok and Hwang around rigidity questions (\cite{HM}, \cite{HM2}). In this article we give the complete list of projective manifolds with quadric structure.

Let $M$ be a compact, connected complex manifold and $S$ a compact hermitian symmetric space of the same complex dimension $m$. There is a notion of a {\em flat} $S$--structure and of an {\em infinitesimal} $S$--structure on $M$. If the universal covering space $\tilde{M}$ of $M$ admits an embedding $\tilde{M}\hookrightarrow S$, such that $\pi_1(M)$ becomes a subgroup of ${\rm HAut}(S)$, then $M$ carries a flat $S$--structure. Typical examples are $M=S$, certain quotients of $\KC^m$ and compact quotients of the non compact dual of $S$ in the sense of hermitian symmetric spaces. 
The infinitesimal version of an $S$--structure asks for a certain subbundle in the holomorphic frame bundle of $M$. Equivalent in the case of $S=\Q_m$, the $m$--dimensional hyperquadric, is the existence of a {\em holomorphic conformal structure}, i.e., a non--degenerate section
  \[g \in H^0(M, S^2T_M \otimes L), \quad L\in \Pic(M).\]
If $M$ carries a quadric structure, then $M$ carries a holomorphic conformal structure, called flat. Whether every holomorphic conformal structure is flat, i.e., comes from a quadric strucutre, is not known.

Any compact Riemann surface carries a flat holomorphic conformal structure by the uniformization theorem. Kobayashi and Ochiai gave a complete list in the compact surfaces case (\cite{KOQuad}), proving flatness in dimension $2$. The authors proved flatness of any holomorphic conformal structure in the case of projective threefolds (\cite{JRquadr}). In the K\"ahler--Einstein case one has the following result (\cite{KOQuad}):

\begin{theorem}
  A compact K\"ahler--Einstein manifold $M_m$ carries a holomorphic conformal structure if and only if $M$ is either
  \begin{enumerate}
    \item $\simeq \Q_m$,
    \item a torus or a conformally hyperelliptic manifold or
    \item a $\D_m$ quotient.
  \end{enumerate}
\end{theorem}
Conformally hyperelliptic manifolds are particular finite {\'e}tale quotients of tori. The notion will be introduced in \ref{ConfHyp}. A $\D_m$ quotient means the universal covering space of $M$ can be identified with $\D_m$, the non compact dual of $\Q_m$ in the sense of hermitian symmetric spaces also reviewed in \ref{Dual} below. We prove:

\begin{theorem}
  Let $M_m$ be a projective manifold with a quadric structure which is \underline{not} K\"ahler--Einstein. Then $M$ is a complex surface and a fiber bundle over a compact Riemann surface $C$. More precisely, there exists a representation
   \[\rho: \pi_1(C) \lra Aut(F), \quad F \mbox{ rational or elliptic,}\]
such that $M = F \times_{\rho} C$.
\end{theorem}
Here $M = F \times_{\rho} C$ means the following. Let $\tilde{C}$ denote the universal covering space of $C$. Then we obtain $M$ by dividing $F \times \tilde{C}$ by the action of $\pi_1(C)$
  \[\gamma(f, \tau)  = (\rho(\gamma)(f), \gamma(\tau)), \quad \gamma \in \pi_1(C).\]
In other words, the only non--K\"ahler--Einstein examples arise in dimension $2$, caused by the fact that $\Q_2 \simeq \PN_1 \times \PN_1$ is not irreducible as a hermitian symmetric space.
  
The above theorem reproves Kobayashi and Ochiai's result in the case of projective surfaces as well as the author's result in the case of projective threefolds and quadric structures. Some of the methods from the proof also apply to the non--flat case, however a general classification for manifolds with holomorphic conformal structures seems out of reach at the moment.

\

\noindent {\bf Notation.} $M$ denotes a compact, connected complex manifold of complex dimension $m$. We always assume $M$ to be K\"ahler.

The holomorphic tangent bundle of $M$ is $T_M$, its dual $\Omega_M^1$, the bundle of holomorphic $1$--forms. The canonical bundle is $K_M = \det \Omega_M^1$. We do not distinguish between Cartier divisors and line bundles. For a line bundle write $rL$ instead of $L^{\otimes r}$. In the case $M$ projective we write $\equiv$ for numerical equivalence of line bundles, i.e. equality of Chern classes in $H^1(M, \Omega_M^1)$. A fibration $M \lra N$ is holomorphic and has connected fibers.

\section{The complex projective hyperquadric and quadric structures}
\setcounter{equation}{0}

There is a general notion of an $S$--structure, $S$ an arbitrary hermitian symmetric space of the compact type. We only need the definition for the spacial case $S = \Q_m$, the $m$--dimensional complex projective quadric, and later for $S = \PN_m(\KC)$. The general constructions are also explained in \cite{KOQuad}, \cite{KNBook}, we include it for the convenience of the reader. 

\subsection{The hyperquadric $\Q_m \subset \PN_{m+1}(\KC)$} Let $q=(q_{ij}) \in M_{m+2,m+2}(\KC)$ be symmetric and non--degenerate. The {\em complex projective quadric} 
  \begin{equation}\label{Qq}
 \Q(q)=\{[Z] = [Z_0: \dots :Z_{m+1}] \in \PN_{m+1}(\KC) \mid \sum q_{ij}Z_iZ_j=0\}
 \end{equation}
is a smooth projective homogeneous variety of complex dimension $m$. Write 
$\Q_m = \Q(I_{m+2})$, where $I_{m+2}$ is the unit matrix, for short. It is well known that $\Q_m\simeq \Q(q)$ for any $q$ by a projective linear change of coordinates. We will use these isomorphisms without further mentioning below. Recall
  \[\Q_1 \simeq \PN_1(\KC) \quad \mbox{and} \quad \Q_2 \simeq \PN_1(\KC)\times \PN_1(\KC).\]
For $m>2$, $\Q_m$ ist Fano with Picard number one. The quadric $\Q_m$ is a hermitian symmetric space of type $BD \, I$. For the later definition of the non--compact dual we have to recall some of the background: the natural $S^1$--bundle on $\Q_m$ 
 \begin{equation} \label{KomplStr}
     \{M \in M_{m+2,2}(\KR) \mid M^tM = I_2\} \lra \Q_m, \quad M \mapsto M\mbox{$\left({i\atop 1}\right)$},
  \end{equation}
identifies $M, M' \in M_{m+2,2}(\KR)$ iff $M = M'D$ for some $D \in SO(2) \simeq S^1$. The columns of such an $M$ define a ON--basis for the standard metric of a twoplane in $\KR^{m+2}$. Geometrically, $M, M' \in M_{m+2,2}(\KR)$ are identified iff they span the same twoplane and with the same orientation. In other words
  \[\Q_m = \OG(2, m+2),\]
the Grassmannian of oriented twoplanes in $\KR^{m+2}$. Recall that with the usual Grassmanian $\G(2, m+2)$ one allows $D \in O(2)$ which gives an {\'e}tale $2:1$ covering $\OG(2, m+2) \lra \G(2, m+2)$. Identifying $\Q_m \simeq \OG(2, m+2)$, one has to divide by the fixed point free action of complex conjugation.

The $S_1$--bundle in \Formel{KomplStr} is compatible with the natural transitive action of the orthogonal group $SO(m+2)$ on both sides by multiplication from the left. Stabilizer of the reference point $(0_{2,m}, I_2)^t$ or $[0:\cdots :0:i:1] \in \Q_m$, respectively, is $SO(m) \times SO(2)$. Hence
   \[\Q_m = \OG(2, m+2) \simeq SO(m+2)/SO(m) \times SO(2).\]
We obtain the description of $\Q_m$ as quotient of the connected compact simple real Lie group $SO(m+2)$ by the maximal connected proper subgroup $SO(m) \times SO(2)$ with non discrete center
\[Z(SO(m) \times SO(2))=\left\{\begin{array}{cc}
SO(m) \times SO(2) & m=1,2 \\
\{I_m\} \times SO(2) & m>2 \\
\end{array}\right.\]
The real tangent bundle on $SO(m+2)/SO(m) \times SO(2)$ is induced by the adjoint representation of $SO(m) \times SO(2)$ on $so(m+2)/so(m) \times so(2)$. The complex structure on $\Q_m$ is defined by $I_m \times D_{\frac{\pi}{2}} \in Z(SO(m) \times SO(2))$, where $D_\alpha \in SO(2)$ denotes the standard positive rotation matrix. The square of the complex structure defines the reflecting automorphism at the reference points. The group of holomorphic automorphisms of $\Q_m$ is $HAut(\Q(S)) \simeq PSO(m+2, \KC)$.

\subsection{Quadric structures}
$M_m$ carries a {\em holomorphic quadric structure}, if there exists a holomorphic atlas $\{U_i, \varphi_i\}_{i\in I}$, such that
\begin{enumerate}
  \item $\varphi_i: U_i \hookrightarrow \Q_m$ is a holomorphic embedding for $i\in I$ and 
  \item $\varphi_i \circ \varphi_j^{-1}$, defined on $\varphi_j(U_{ij}) \subset \Q_m$, is the restriction of some automorphism from $\HAut(\Q_m) = PSO(m+2,\KC)$.
\end{enumerate}
Clearly $\Q_m$ carries a quadric structure. If $M$ carries a quadric structure and $M' \lra M$ is \'etale, then $M'$ carries a quadric structure. A major source of examples is

\begin{lemma}\label{UnivCov} If there exists an embedding of the universal covering space $\tilde{M} \hookrightarrow \Q_m$, such that $\pi_1(M)$ acts on $\tilde{M}$ by restricting automorphisms from $\HAut(\Q_m)$, then $M$ carries a quadric structure.
\end{lemma}
The map $SO(m+2, \KC)\lra PSO(m+2, \KC)$ is $2:1$ \'etale. The line bundle $\O_{\Q_m}(d)=\O_{\PN_{m+2}(\KC)}(d)|_{\Q_m}$ is $SO(m+2, \KC)$--homogeneous, and $PSO(m+2, \KC)$--homogeneous in the case $d$ even. Hence, for $d$ even, $\O_{\Q_m}(d)$ induces a line bundle on any manifold $M$ with a quadric structure.

\subsection{Development} (\cite{KO}, \cite{KoWu}) Let $M_m$ carry a quadric structure, let $\{(U_i, \varphi_i)\}_{i\in I}$ be an atlas as above. Choose one coordinate chart $(U_0, \varphi_0)$. Given a point $p \in M_m$, choose a chain of charts $(U_i, \varphi_i)_{i=0, \dots, r}$ such that $U_i \cap U_{i-1} \not= \emptyset$, $i = 1, \dots, r$ and $p \in U_r$. We have
 \[\varphi_{i-1}\circ \varphi_i^{-1} = g_i|_{\varphi_i(U_i\cap U_{i-1})}\]
for some $g_i \in HAut(\Q_m)$. Set $\psi(p) := (g_1 \circ g_2 \circ \cdots \circ g_r\circ \varphi_r)(p) \in \Q_m$. This gives a multivalued map $\psi: M \to \Q_m$ which is single valued in the case $M$ simply connected. It is called a {development} of $M$ to $\Q_m$. By construction, $\psi$ is immersive.

If $M$ is compact and simply connected, then $M \simeq \Q_m$. If $M$ is not simply connected, then let $\tilde{M}$ be its universal covering space and let $\psi: \tilde{M} \to \Q_m$ be a development of $\tilde{M}$. We get a natural map $\rho: \pi_1(M) \to HAut(\Q_m)$ such that $\psi(\gamma(p)) = \rho(\gamma)(\psi(p))$ for any $\gamma \in \pi_1(M)$ (\cite{KoWu}). 

\subsection{Dual space $\D_m$} \label{Dual} In a fixed basis of $\KR^{m+2}$, let $I_{m,2}$ be the symmetric form of signature $(m \times -, 2 \times +)$. The real Lie group
  \[SO(m,2) = \{A \in Sl_{m+2}(\KR) \mid A^tI_{m,2}A = I_{m,2}\}\]
is well known to have two connected components, we denote the one containing the identity $SO^0(m,2)$. The symmetric space $BDI_{m,2}$ dual to $\OG(2, m+2)$ parametrizes twoplanes in $\KR^{m+2}$ on which $I_{m,2} > 0$ is positive definit. Hence
 \[BDI_{m,2}  = \{M \in M_{m+2,2}(\KR) \mid M^tI_{m,2}M >0\}/\sim,\]
where $M \sim M'$ iff $M = M'D$ for some $D \in Gl_2(\KR)$. The twoplanes are generated by the columns of $M$. The group $SO(m,2)^0$ acts transitively on $BD \, I_{m,2}$ from the left, stabilizer of $(0_{m,2}, I_2)^t$ is $SO(m,2)^0 \cap S(O(m)\times O(2)) = SO(m) \times SO(2)$ and
 \[BDI_{m,2}  \simeq  SO(m,2)^0/SO(m) \times SO(2).\]
If we write $M = \left({M_1 \atop M_2}\right)$, $M_1 \in M_{m, 2}(\KR)$, then $M^tI_{m,2}M > 0$ implies $M_2 \in Gl_2(\KR)$. As $M \sim MM_2^{-1}$ we find the classical description
 \[BDI_{m,2}  \simeq \{m \in M_{m,2}(\KR) \mid m^tm < I_2\},\]
with the action of $SO(m,2)^0$ given by
  \[m \mapsto (Am+B)(Cm+D)^{-1}, \quad \left(\begin{array}{cc}
                                          A_{m,m} & B_{m,2} \\
                                          C_{2,m} & D_{2,2}
\end{array}\right) \in SO(m,2)^0.\]
By choosing an ON--Basis, any twoplane on which $I_{m,2}>0$ is represented by  $M = \left({M_1 \atop M_2}\right) \in M_{m+2, 2}(\KR)$ such that $M^tI_{m,2}M =I_2$ and $\det M_2>0$. In this description, $M, M' \in M_{m+2,2}(\KR)$ define the same twoplane iff there exists $D \in SO(2)$ such that $M = M'D$ and $M \mapsto M_i=M\left({i \atop 1}\right)$ gives
  \[BDI_{m, 2} \hookrightarrow \Q(I_{m,2}).\]
This identifies $BDI_{m, 2}$ and the open subset of $\Q(I_{m,2})$
  \[\D_m = \{[W] \in \Q(I_{m,2}) \mid W^tI_{m,2}\bar{W} > 0, i(\bar{W}_{m+1}W_{m+2}-W_{m+1}\bar{W}_{m+2})>0\}.\]
The group of holomorphic automorphisms of $\D_m$ is
 \[\HAut(\D_m) = PSO^0(m,2) = \{f \in \HAut(\Q(I_{m+2}) \mid f(\D_m) = \D_m\}.\]
By \Lem{UnivCov},

\begin{proposition}
  Any complex manifold whose universal covering space is $\D_m$ carries a quadric structure.
\end{proposition}

In order to see that $\D_m$ is a bounded symmetric domain let $Z_{m+1} = \frac{1}{\sqrt{2}}(W_{m+1} - iW_{m+2})$, $Z_{m+2} = \frac{1}{\sqrt{2}}(W_{m+1} + iW_{m+2})$ and $Z_k = W_k$ for $k=1,\dots,m$. In these new coordinates the defining equation for $\Q(I_{m,2})$ becomes 
  \[2Z_{m+1}Z_{m+2} = Z_1^2 + \cdots + Z_m^2,\]
i.e. the change of coordinates transforms $\Q(I_{m,2})$ into $\Q(q_0)$, where 
\[q_0 = \left(\begin{array}{ccc}
                    I_m & 0 & 0 \\
                    0 & 0 & -1 \\
                    0 & -1 & 0
                   \end{array}\right).\]
The intersection $\Q(q_0) \cap \{Z_{m+2} \not= 0\}$ is isomorphic to $\KC^m$ by
 \[j: \KC^m \lra \Q(q_0), \quad z \mapsto [z:\mbox{$\frac{1}{2}$}z^tz:1].\]
Now $\D_m = \{[Z] \in \Q(q_0) \mid  Z^tI_{m,2}\bar{Z} > 0, |Z_{m+1}|^2 < |Z_{m+2}|^2\} \subset \Q(q_0) \cap \{Z_{m+2} \not= 0\}$ and one checks $\D_m = j(\{z \in \KC^m \mid z^t\bar{z} < 1+|\frac{1}{2}z^tz|^2 < 2\})$.
 
\subsection{Conformally hyperelliptic varieties} \label{ConfHyp}
(see also \cite{KOQuad}, \cite{KoWu}) Let $\COAff_m(\KC)$ be the group of conformal affine maps
  \[\sigma_{A,b,c}: \KC^m \lra \KC^m, \;\; \sigma_{A,b,c}(z) = cAz+b, \quad A \in O(m, \KC), b \in \KC^m, c \in \KC^*.\]
A direct computation shows $\rho: \COAff_m(\KC) \lra O(q_0, \KC)$ given by
\[\rho(\sigma_{A,b,c}) =  \left(\begin{array}{ccc}
            A & 0_{m,1} & \frac{1}{c}b \\
            b^tA & c & \frac{1}{2c}b^tb \\
             0_{1,m} & 0 & \frac{1}{c}
\end{array}\right)\]
is a group morphism and $j(\sigma_{A,b,c}(z)) = \rho(\sigma_{A,b,c})(j(z))$ with $j$ from above.

We will call a compact complex manifold $M$ {\em conformally hyperellitpic}, if $M$ is not a torus but isomorphic to a quotient of $\KC^m$ by some fixed point free subgroup $\pi_1 \subset COAff_m(\KC)$ which contains a complete lattice $\Lambda$ as a subgroup of finite index. 

Such an $M$ is a finite {\'e}tale quotient of a torus. By \Lem{UnivCov}, using $j$ and $\rho$ we have
\begin{proposition} \label{TorHyp}
  Tori and conformally hyperelliptic manifolds carry a quadric structure.
\end{proposition}

\section{Holomorphic conformal structures}
\setcounter{equation}{0}

A complex manifold $M$ is said to carry a {\em holomorphic conformal structure} if there exists a non degenerate holomorphic
  \[g \in H^0(M, S^2\Omega_M \otimes L^*),\]
where $L \in \Pic(M)$ is some line bundle. An equivalent definition involves a $CO(m, \KC) = \{cA \mid A \in O(m, \KC), c \in \KC^*\}$ bundle in the frame bundle of $M$, see \cite{KOQuad}.

We will use the same letter $g$ to denote the induced map from $T_M \otimes L^*$ to $\Omega_M^1$ by $g$. By assumption $g$ gives an isomorphism
  \begin{equation} \label{LK}
      g: \;\; T_M \otimes L^* \simeq \Omega_M^1 \quad \mbox{and} \quad -2K_M \simeq mL.
  \end{equation}
On the quadric $\Q(q)$ in \Formel{Qq}, $\sum_{ij}q_{ij}dZ_idZ_j$ induces a holomorphic conformal structure with the choice $L = \O_{\PN_{m+1}(\KC}(-2)|_{\Q_m}$. It is invariant under the action of $\HAut(\Q_m)$. A manifold $M$ with a quadric structure carries a holomorphic conformal structure. We will call the conformal structure {\em flat} in this case.

\subsection{Atiyah class of $\Omega_M^1$} The {\em Atiyah class} $b(E) \in H^1(M, \Hom(E,E)\otimes \Omega_M^1)$ associated to a holomorphic vector bundle $E$ on a complex manifold $M$ (\cite{At}) is the obstruction to the holomorphic splitting of the first jet sequence
  \[0 \lra E\otimes \Omega_M^1 \lra J_1(E) \lra E \lra 0.\]
Let $\U = \{U_{\alpha}\}$ be a trivializing atlas, i.e., $E|_{U_{\alpha}} \simeq \KC^r \times U_{\alpha} =: E_{\alpha}$. Denote the induced transition functions by $t_{\alpha\beta}: E_{\beta} \lra E_{\alpha}$, defined over $U_{\alpha\beta}$, i.e., $t_{\alpha \beta} \in Gl_{m}(\O_{U_{\alpha,\beta}})$. Then the matrix of holomorphic oneforms $dt_{\alpha\beta}$ is a well defined $1$--cocycle of $\Hom(E,E)\otimes \Omega_M^1$, representing $b(E)$. Under Dolbeault $H^1(M, \Hom(E,E) \otimes \Omega_M^1) \simeq H^{1,1}(M, \Hom(E,E))$, $b(E)$ corresponds to $[\Theta_h]$, where $\Theta_h$ is the curvature tensor associated to some hermitian metric $h$ on $E$. In the K\"ahler case the first Chern class of $E$ is given by $tr (b(E)) = -\frac{1}{2i\pi}c_1(E) \in H^1(M, \Omega_M^1)$. This is why we normalize $a(E) := -\frac{1}{2i\pi} b(E)$.

Consider $E=T_M$. Let $\U = \{(U_{\alpha}, z_{\alpha}^1, \dots, z_{\alpha}^m)\}$ be an atlas. Transition functions of $T_M$ are $t_{\alpha\beta} = (\frac{\partial z_{\alpha}^i}{\partial z_{\beta}^q})_{i,q}$ and $b(T_M)$ is represented by 
  \[\sum_{iqr}\frac{\partial^2 z_{\alpha}^i}{\partial z_{\beta}^q\partial z_{\beta}^r} \frac{\partial}{\partial z_{\alpha}^i}\otimes dz_{\beta}^q \otimes dz_{\beta}^r = \sum_{ipqr}\frac{\partial^2 z_{\alpha}^i}{\partial z_{\beta}^q\partial z_{\beta}^r}\frac{\partial z_{\beta}^p}{\partial z_{\alpha}^i}\frac{\partial}{\partial z_{\beta}^p} \otimes dz_{\beta}^q \otimes\otimes dz_{\beta}^r.\]

Let $M$ be compact K\"ahler with a holomorphic conformal structure. Then $c_1(K_M) \in H^1(M, \Omega_M^1)$ and via $g: T_M\otimes L^* \lra \Omega_M^1$
  \[g^{-1}(c_1(K_M)) \in H^1(M, T_M \otimes L^*).\]
Then $g \otimes g^{-1}(c_1(K_M)) \in H^1(M, S^2\Omega_M \otimes L \otimes T_M \otimes L^*) \simeq H^1(M, S^2\Omega_M \otimes T_M)$. The following Proposition is a consequence of standard formulas from Riemannian geometry interpreted in the holomorphic setting.
\begin{proposition} \label{AtCl}
  On a compact K\"ahler manifold with a holomorphic conformal structure $g \in H^0(M, S^2\Omega_M^1 \otimes L)$, the normalized Atiyah class of $\Omega_M^1$ is given by
   \begin{equation} \label{AtClHCS}
a(\Omega^1_M) =  \frac{c_1(K_M)}{m} \otimes id + id \otimes \frac{c_1(K_M)}{m} - g \otimes g^{-1}(\frac{c_1(K_M)}{m})
  \end{equation}
in $H^1(M, \Omega_M^1 \otimes T_M \otimes \Omega_M^1)$.
\end{proposition}
Note $\Omega_M^1 \otimes T_M \otimes \Omega_M^1 \simeq \Omega_M^1 \otimes {\rm End}(T_M) \simeq {\rm End}(T_M) \otimes \Omega_M^1$. \Prop{AtCl} implies certain relations among the Chern classes of $M$. For example 
  \begin{equation} \label{C2}
    m^2c_2(M) = ({m \choose 2} + 1)c_1^2(M) \;\; \mbox{$H^2(M, \Omega_M^2)$}.
  \end{equation}
Roughly speaking any algebraic relation among Chern classes on the quadric holds on $M$. For details see \cite{KOQuad}. 

\begin{proof}
  Choose an atlas $\U = \{(U_{\alpha}, z_{\alpha}^1, \dots, z_{\alpha}^m)\}$ such that, locally, $g$ is given by 
  \[g_{\alpha} = \sum_{i,j}g_{\alpha ij}dz_\alpha^i\otimes dz_\alpha^j, \quad g_{\alpha ij} \in \O_M(U_\alpha)\]
with $g_{\alpha ij} = g_{\alpha ji}$ and $g_{\alpha} = l_{\alpha\beta}g_{\beta}$. Here $l_{\alpha\beta} \in \O_{U_{\alpha\beta}}^*$ is a $1$--cocycle defining $L$. As a map $T_M \lra \Omega_M^1 \otimes L^*$ in the chosen frame, $g$ is represented by the matrix $g_{\alpha ij}$, the inverse is denoted $(g_{\alpha}^{ij})$. As $2K_M = mL$, $\frac{1}{2}\sum_h \frac{\partial l_{\alpha\beta}}{\partial z_{\beta}^h}dz_{\beta}^h$ represents $c_1(K_M)/m$ in $H^1(M, \Omega_M^1)$. We have to show that there is a cocylce solution to
 \[\sum_{pqr}\left(\sum_i \frac{\partial z_\beta^p}{\partial z_\alpha^i}\frac{\partial^2 z_\alpha^i}{\partial z_\beta^q \partial z_\beta^r} -  \frac{1}{2}\delta_{pq}
\frac{\partial l_{\alpha\beta}}{\partial z_{\beta}^r} -\frac{1}{2} 
\delta_{pr}\frac{\partial l_{\alpha\beta}}{\partial z_{\beta}^q} +\frac{1}{2}
g_{\beta}^{ph} \frac{\partial l_{\alpha\beta}}{\partial z_{\beta}^h}g_{\beta qr}\right)\frac{\partial}{\partial z_{\beta}^p}\otimes dz_{\beta}^q\otimes dz_{\beta}^r.\]
This is a consequence of standard formulas from Riemannian geometry: as in the Riemannian case we associate to $g_{\alpha}$ a holomorphic affine connection on $U_{\alpha}$ by
   \begin{equation} \label{Chris}
     \Gamma_{jk}^i = \Gamma_{jk}^i(g_{\alpha}) = \frac{1}{2}\sum_h g_{\alpha}^{ih}\left(\frac{\partial g_{\alpha hj}}{\partial z_{\alpha}^k} + 
\frac{\partial g_{\alpha hk}}{\partial z_{\alpha}^j} -
\frac{\partial g_{\alpha jk}}{\partial z_{\alpha}^h}\right).
   \end{equation}
Let $\tilde{g}_{\beta} = l_{\alpha\beta}g_{\beta}$. Then $g_\alpha = \tilde{g}_\beta$ on $U_{\alpha\beta}$ as in the Riemannian case. Therefore, if we put $\tilde{\Gamma}_{qr}^p = \Gamma_{qr}^p(\tilde{g}_{\beta})$, then we have the well known formula of base change 
for the Christoffel symbols: 
  \[\tilde{\Gamma}_{\beta qr}^p = \sum_{ijk}\frac{\partial z_\alpha^j}{\partial z_\beta^q} \frac{\partial z_\alpha^k}{\partial z_\beta^r} \Gamma_{\alpha jk}^i
\frac{\partial z_\beta^p}{\partial z_\alpha^i}+
\sum_i\frac{\partial z_\beta^p}{\partial z_\alpha^i}\frac{\partial^2 z_\alpha^i}{\partial z_\beta^q \partial z_\beta^r}.\]
Using product rule and $g_\beta g_\beta^{-1}=id$ we may express $\tilde{\Gamma}_{qr}^p$ in terms of $\Gamma_{qr}^p(g_{\beta})$:
  \[\tilde{\Gamma}_{\beta qr}^p = \frac{1}{2}\sum_h l_{\alpha\beta}^{-1}\left(\delta_{pq}
\frac{\partial l_{\alpha\beta}}{\partial z_{\beta}^r} + 
\delta_{pr}\frac{\partial l_{\alpha\beta}}{\partial z_{\beta}^q} -
g_{\beta}^{ph} \frac{\partial l_{\alpha\beta}}{\partial z_{\beta}^h}g_{\beta qr}\right) + \Gamma_{qr}^p,\]
where $\Gamma_{qr}^p = \Gamma_{qr}^p(g_{\beta})$. Then $\sum_{ijk}\Gamma_{jk}^i \frac{\partial}{\partial z_{\alpha}^i}\otimes dz_{\alpha}^j\otimes dz_{\alpha}^k \in Z(\U, T_M\otimes\Omega_M^1\otimes \Omega_M^1)$ is the desired cocycle solution.
\end{proof}

\subsection{Extensions}
Let $\Omega$ be an arbitrary vector bundle on some complex manifold $M$. Let $L \in \Pic(M)$ be some line bundle and $g \in H^0(M, S^2\Omega \otimes L)$. Assume there is a vector bundle $F$, $rk F > 0$, and an exact sequence 
  \begin{equation} \label{SeqE} 
    0 \lra F \lra \Omega \lra F^*\otimes L^* \lra 0.
  \end{equation}
Dualize and multiply by $L^*$. We say $g$ {\em is compatible with \Formel{SeqE}}, if we obtain a commuting diagram
 \[\xymatrix{0 \ar[r] & F \ar[d]^{\varphi} \ar[r] & \Omega^*\otimes L^* \ar[d]^{g} \ar[r] & F^*\otimes L^* \ar[d]^{\psi}\ar[r] & 0 \\
0 \ar[r] & F \ar[r] & \Omega \ar[r] & F^*\otimes L^* \ar[r] & 0.}\]
Obviously, $g$ is compatible in the case $\Hom(F, F^*\otimes L) = 0$. A compatible $g \in H^0(M, S^2\Omega \otimes L)$ is non--degenerate if and only if $\varphi$ and $\psi$ are isomorphisms.

\begin{lemma} \label{CSeq}
  Let $e \in H^1(M, F \otimes F \otimes L)$ be the extension class of \Formel{SeqE}. Assume $\Hom(F, F) \simeq \KC$. Then there exists a non--degenerate $g \in H^0(M, S^2\Omega \otimes L)$ compatible with \Formel{SeqE} if and only if $e \in H^1(M, \bigwedge^2 F \otimes L)$.
\end{lemma}
Here we use the canocial isomorphism $F \otimes F \simeq S^2F \oplus \bigwedge^2 F$. 
Note that in the particular case $rk F = 1$ the condition $e \in H^1(M, \bigwedge^2 F \otimes L)$ implies the holomorphic splitting of \Formel{SeqE}.

\begin{proof}
 From \Formel{SeqE} we obtain
  \[0 \lra K \otimes L \lra S^2\Omega \otimes L \lra S^2F^* \otimes L^* \lra 0\]
  \[0 \lra S^2F \otimes L \lra K \otimes L \lra F \otimes F^* \lra 0.\]
The set of compatible $g$'s is $H^0(M, K\otimes L) \subset H^0(M, S^2\Omega \otimes L)$. A section coming from $H^0(M, S^2F \otimes L)$ will be degenerated. 

Since $\Hom(F,F) \simeq \KC$, a compatible non--degenerate $g$ exists if and only if we can lift $id_{F} \in H^0(M, F \otimes F^*)$ to some $g \in H^0(M, K\otimes L)$. Such a $g$ exists if and only id $id_{F}$ is mapped to zero by the the coboundary map $\delta:  H^0(M, F \otimes F^*) \lra H^1(M, S^2F\otimes L)$. Now $\delta(\id_F)$ is the image of $e$ under $H^1(M, F\otimes F \otimes L) \lra  H^1(M, S^2F\otimes L)$. Therefore $g$ exists if and only if $e \in H^1(M, \bigwedge^2 F \otimes L)$.
\end{proof}

\section{Hyperelliptic manifolds}
\setcounter{equation}{0}

A compact complex manifold $M_m$ is called {\em hyperelliptic} if $M$ is a finite {\'e}tale quotient of a complex torus $T = \KC^m/\Lambda$. The universal covering space of such an $M$ can be identified with $\KC^m$ such that $\pi_1(M) \subset \HAut(\KC^m)$ can be identified with a subgroup of $\Aff_m(\KC)$, the group of affine maps $\sigma_{A,b}: z \mapsto Az+b$, $A\in Gl_m(\KC)$, $b\in \KC^m$. 

The following description can essentially be found in \cite{KOQuad}:

\begin{theorem}
  The following conditions on a hyperelliptic manifold $M_m$ are equivalent:
  \begin{enumerate}
    \item $M$ carries a quadric structure.
    \item $M$ carries a holomorphic conformal sructure.
    \item $M$ is conformally hyperelliptic.
   \end{enumerate}
\end{theorem}

\begin{proof}
  $1.) \Rightarrow 2.)$ is true for any manifold, $3.) \Rightarrow 1.)$ is \Prop{TorHyp}. We prove  $2.) \Rightarrow 3.)$.

By assumption, $M$ is the quotient of $\KC^m$ by some subgroup of $\Aff_m(\KC)$ which contains a complete lattice $\Lambda$ as a normal subgroup of finite index. The lattice defines a torus $T$ and a finite {\'e}tale Galois covering $\pi: T \lra M$. Let $g \in H^0(M, S^2\Omega_M^1 \otimes L)$ be the holomorphic conformal structure. Then  $g_T = \pi^*g \in H^0(T, S^2\Omega_T^1 \otimes \pi^* L)$ is non degenerate. By \Formel{LK} $\pi^*L$ is torsion on $T$, because $K_T$ is trivial. After some finite {\'e}tale covering $T'\lra T$ to another torus $T'$, $L$ becomes trivial. Another covering is Galois, so we may assume that already $\pi^*L = \O_T$. 

Then $L$ is induced by some representation $\rho: \pi_1(M) \lra \KC^*$, which is trivial on $\Lambda$. We may choose coordinates such that $g_T$ is induced by $dz_1^2 + \cdots +dz_m^2$ on $\KC^m$. Then $g_T$ descends to a section of $S^2\Omega_M^1 \otimes L$ if and only of $\sigma_{A,b}^*\omega_T = \rho(\sigma_{A,b})\omega_T$. This is the case if and only if $\sigma_{A,b} \in \COAff_m(\KC)$.
\end{proof}

\begin{example}
1.) Let $E,F$ be elliptic curves. Let $U \subset E$ be some nontrivial subgroup and $\rho: U \lra Aut(F)$ be injective. The group $U$ acts on $E \times F$ without fixed points via $(z_1, z_2) \mapsto (z_1 + u, \rho(u)(z_2))$. The quotient $M = E \times F/U$ is a hyperelliptic surface with a holomorphic conformal structure induced by $dz_1dz_2$.

2.)  Let $E_1, E_2$ be elliptic curves, $F = \KC/\KZ + \KZ\xi$ where $\xi$ is some third root of unity. Let $e_6\in E_1[6]$ be a point of order $6$. Consider the fixed point free action of $\KZ_6$ on $T = E_1 \times E_2 \times F$:
   \[(z_1, z_2, z_3) \mapsto (z_1 + e_6, -z_2, \xi z_3).\]
It is easy to see that every eigenvector of the induced action on $H^0(T, S^2\Omega_T^1)$ is degenerate. The quotient $M =  E_1 \times E_2 \times F/\KZ_6$ is a hyperelliptic manifold that does not admit a quadric strucutre.
\end{example}

\section{Classification} \label{birklass} 
\setcounter{equation}{0}

The Theorem by Kobayashi and Ochiai from the introduction says a compact K\"ahler--Einstein manifold $M_m$ carries a holomorphic conformal structure if and only if up to isomorphisms $M$ is either
  \begin{enumerate}
    \item $\Q_m$,
    \item a torus or a conformally hyperelliptic manifold or
    \item a $\D_m$ quotient.
  \end{enumerate}
Recall that K\"ahler--Einstein implies $\pm c_1(M) > 0$ or $c_1(M) = 0$. Conversely, by results of Aubin \cite{Au} and Yau \cite{Yau}, if $c_1(M) < 0$ or $c_1(M)=0$, then $M$ admits a K\"ahler--Einstein metric.

\subsection{Rational curves} Let $M_m$ be a projective manifold with a quadric structure. If $K_M$ is not nef, then $M$ contains a rational curve by the cone theorem (\cite{KM}) meaning that there exists a non--constant holomorphic map
  \[\nu: \PN_1(\KC) \lra M.\]

\begin{proposition} \label{ratcurve}
  Let $M_m$ be a projective manifold with a quadric structure. If $M$ contains a rational curve then $M$ is uniruled and either
\begin{enumerate} 
 \item $M \simeq Q_m$ or
 \item $M$ is ruled surface over a compact Riemann surface $B$ of genus $g_B > 0$ induced by some representation $\rho: \pi_1(B) \lra PGl_2(\KC)$.
\end{enumerate}
\end{proposition}
In the second point, $M = \PN_1(\KC) \times \tilde{B}/\pi_1(B)$, where $\gamma \in \pi_1(B)$ acts via $(\tau, b) \mapsto (\rho(\gamma)(\tau), \gamma(b))$.

\begin{corollary} \label{ratcurveC}
  Let $M_m$ be a projective manifold with a quadric structure. If $M$ does not contain a rational curve, then $K_M$ is nef. 
\end{corollary}
Recall that $K_M$ nef means $K_M.C\ge 0$ for any curve $C$ in $M$. The corollary is an immediate consequence of the proposition and the cone theorem (\cite{KM}).

\begin{proof}[Proof of Proposition~\ref{ratcurve}.] Let $\psi: \tilde{M} \lra Q_m$ be a development of the universal covering space $\mu: \tilde{M} \lra M$. We have an induced map $\tilde{\nu}: \PN_1(\KC) \lra \tilde{M}$ such that $\mu \circ \tilde{\nu} = \nu$. Homogeneity implies $T_{Q_n}$ is spanned by global sections. Then $\nu^*T_M = (\psi \circ \tilde{\nu})^*T_{Q_m}$ is nef, i.e., $\nu^*T_M = \O_{\PN_1}(a_1) \oplus \cdots \oplus \O_{\PN_1}(a_m)$ where $a_k \ge 0$ for all $k$. Then $M$ contains a free rational curve. Then $M$ is uniruled, i.e., covered by rational curves, by \cite{KollBuch}, IV, 1.9.Theorem.

In the case $m \not= 2$ the quadric $\Q_m$ is irreducible as a hermitian symmetric space. Then $M$ uniruled implies $M \simeq Q_m$ by \cite{HM}. 

In the case $m = 2$, $M$ is a projective surface. As $T_M$ is nef on every rational curve, $M$ is free of $(-1)$--curves. The list of minimal uniruled projective surfaces is $\PN_2$ or a ruled surface over some smooth base curve $B$. 

As $9 = c_1(\PN_2)^2 = 3c_2(\PN_2)$ implying $c_2(\PN_2) \not= \frac{1}{2}c_1(\PN_2)^2$ we infer $M \not= \PN_2$. On a ruled surface $\pi: M \lra B$ we have the exact sequence
  \begin{equation} \label{seq0}
    0 \lra \pi^*K_B \lra \Omega_M^1 \lra K_{M/B} \lra 0.
  \end{equation}
The section $g$ is necessarily compatible with this extension, since $\Hom(\pi^*K_B, \pi^*T_B\otimes K_M) = H^0(B, T_B^{\otimes 2} \otimes \pi_*K_M) = 0$. By \Lem{CSeq} the sequence splits holomorphically. This means the fiber bundle $M \lra B$ has a connection. Then $M$ is induced by a representation of the fundamental group $\pi_1(B) \lra Aut(\PN_1) = PGl_2(\KC)$.
\end{proof}

\subsection{Large fundamental group}
Recall from  (\cite{Ko}, 1.7.) that $M$ is said to have {\em large fundamental group} if for any irreducible complex subvariety $W \subset M$ of positive dimension
  \[{\rm Im}[\pi_1(W_{norm}) \lra \pi_1(M)] \quad \mbox{is infinite.}\]
Here $W_{norm}$ denotes the normalization of $W$.

\begin{proposition} \label{lfg}
  A projective manifold $M$ with a quadric structure which is not uniruled has large fundamental group.
\end{proposition}

\begin{proof} If $M$ is not uniruled $K_M$ is nef (\Cor{ratcurveC}).
Assume to the contrary ${\rm Im}[\pi_1(W_{norm}) \to \pi_1(M)]$ is finite for some $W$ as above. Denote the normalization map by $\nu: W_{norm} \to W$.

Let $C_{norm}$ be some general curve in $W_{norm}$, i.e., the intersection of $\dim W - 1$ general hyperplane sections. Think of $C_{norm}$ as the normalization of $C := \nu(C_{norm}) \subset W$. We have
  \[\pi_1(C_{norm}) \lra \pi_1(W_{norm}) \lra \pi_1(M).\]
Then ${\rm Im}[\pi_1(C_{norm}) \lra \pi_1(M)]$ is finite. The kernel of $\pi_1(C_{norm}) \lra \pi_1(M)$ induces a finite {\'e}tale covering $C' \lra C_{norm}$ from a compact Riemann surface $C'$, such that $\mu: C' \lra C \subset M$ factors over $\tilde{M}$, the universal covering space of $M$. Let $\psi: \tilde{M} \lra  Q_m$ be a development. Denote the induced map $C' \lra \tilde{M} \lra Q_m$ by $\psi_1$. Then $\mu^*T_M = \psi_1^*T_{Q_m}$, implying $\deg(\mu^*(-K_M)) = \deg \psi_1^*(-K_{Q_m}) > 0$. This contradicts $K_M$ nef.
\end{proof}

\subsection{Minimal case}
Recall the following basic facts from the minimal model program: let $M$ be a {\em minimal} projective manifold, i.e., assume $K_M$ nef. The {\em abundance conjecture} asserts that $|dK_M|$ is spanned for some $d \gg 0$, defining the {\em Iitaka fibration}
  \[\Phi: \; M \lra Y,\]
onto some normal projective $Y$ of dimension $\dim Y = \kappa(M)$. Here $\kappa$ denotes the {\em Kodaira dimension}, $\Phi$ has connected fibers. The abundance conjecture is known to hold true for $\dim M \le 3$ and in the case $K_M$ big (meaning $K_M^{\dim M}>0$) by the base point free theorem. It is open in higher dimensions.

If $M$ does not contain any rational curve, then $M$ is minimal by the cone theorem (\cite{KM}). Moreover, in that case any rational map
  \[\xymatrix{M' \ar@{..>}[r] & M}\]
from a manifold $M'$ must be holomorphic. Indeed, if we first had to blow up $M'$ in order to make it holomorphic, then we would find a rational curve in $M$. A manifold with this property is sometimes called {\em strongly minimal} or {\em absolutely minimal} (\cite{MoriClass}, \S 9).

\begin{theorem} \label{ClassAbSc}
Let $M$ be a projective manifold with a quadric structure. Assume that $M$ does not contain a rational curve. Then there exists a finite \'etale covering $M' \lra M$, such that $M'$ is an abelian group scheme over some projective manifold $N'$ such that $K_{N'}$ is ample.
\[\xymatrix{M' \ar[r] \ar[d] & M\\
N' & }\]
\end{theorem}

Here $M'$ an abelian group scheme over $N'$ means there exists a smooth surjective holomorphic map $f: \; M'\lra N'$ such that every fiber is an abelian variety. Moreover, $f$ admits a section.

\begin{proof}
The proof is essentially as in the case of a flat projective structure (\cite{JRpcgen}, \S 3), we therefore only scetch the idea.
  
  Since $M$ does not contain any rational curve, $K_M$ is nef. The quadric structure gives a non--trivial representation of $\pi_1(M)$. We obtain a non--trivial rational map $\xymatrix{M \ar@{..>}[r] & Y}$ with $Y$ of general type (\cite{Zuo}, \cite{Kaw85}). Using \cite{Lai} we conlude $K_M$ abundant.
  
Let $\Phi: M \lra Y$ be the Iitaka fibration with general fiber $F$. Then $N_{F/M}\simeq \O_F^{\oplus \dim M - \dim Y}$, implying $K_F = K_M|_F + \det N_{F/M} = K_M|_F$ is torsion in $\Pic(F)$. It induces a finite {\'e}tale covering $\tilde{F} \to F$ such that $K_{\tilde{F}}$ is trivial. 

By Beauville's decomposition theorem (\cite{Bo}, \cite{Be}) we have, perhaps after another \'etale covering, 
   \[\tilde{F} \simeq A_y \times B_y\]
where $A_y$ is abelian and $B_y$ is simply connected. By \Prop{lfg}, $\pi_1(B_y) = \{id\}$ implies $B_y$ is a point. Therefore, $F$ is an \'etale quotient of an abelian variety.

By (\cite{Ko}, 6.3.~Theorem), there exists a finite {\'e}tale covering $M' \lra M$, s.t. $M'$ is {\em birational} to an abelian group scheme over some base $Y'$ obtained as Stein factorization of $M' \lra M \lra Y$. 

Since $M'$ cannot contain a rational curve, it is strongly minimal and the birational map must be an isomorphism. Then $M' \lra Y'$ is an abelian group scheme. The map has a section. Then $Y'$ does not contain a rational curve. Then $K_{Y'}$ is nef.

By \cite{Ko}, Proposition~5.9., $\kappa(M') = \kappa(Y')$. For any line bundle $L$ on a projective manifold, $\kappa(L) \le \nu(L) := \min\{k \mid L^k \equiv 0\}$. Then $\nu(K_{Y'}) \ge \kappa(Y') = \kappa(M') = \dim Y'$ implies $K_{Y'}$ big. By the base point free theorem (\cite{KM}), $|dK_{Y'}|$ is spanned for some $d \gg 0$, inducing the Iitaka fibration 
$Y' \lra Z$ which is birational in this case. Fibers of positive dimension are covered by rational curves by \cite{Kaw}, Theorem~2. Then $Y' \lra Z$ is finite and $K_{Y'}$ is ample.
\end{proof}

\section{Abelian group schemes}
\setcounter{equation}{0}

Because of \Theo{ClassAbSc} we now consider the following situation: $f: M \lra N$ is an abelian group scheme, $K_N$ is ample. We exclude the possibility of $M$ being K\"ahler--Einstein by assuming
  \begin{equation} \label{dimassum}
     0 < n=\dim_{\KC} N < m=\dim_{\KC} M.
  \end{equation}
The bundle of relative $1$--forms $\Omega_{M/N}^1$ is the pullback of the rank $m-n$ bundle $E^{1,0} = f_*\Omega_{M/N}^1$, $E^{0,1}=R^1f_*\O_M \simeq E^{1,0*}$. We have the exact sequence
  \[0 \lra f^*\Omega_N^1 \lra \Omega^1_M \lra f^*E^{1,0} \simeq \Omega_{M/N}^1 \lra 0.\]
When restricted to $N$ this is the dualized tangent sequence of $N$ and $E^{1,0} = N_{N/M}^*$ under this identification. The connecting morphism 
  \[f_*f^*E^{1,0} \simeq E^{1,0} \lra  R^1f_*f^*\Omega_N^1 \simeq \Omega_N^1 \otimes R^1f_*\O_M \simeq \Omega_N^1 \otimes E^{0,1}\]
is the Kodaira Spencer map. We want to prove:

\begin{proposition} \label{main}
 Let $M \lra N$ be as above. Then $M$ has a holomorphic conformal structure if and only if
$M$ is a complex surface and
   \[M \simeq F \times_{\rho} C\]
 where $F$ is an elliptic curve, $C$ is a curve of positive genus and $\rho: \pi_1(C) \lra Aut(F)$ is some representation.
\end{proposition}

The proof is given below. First:
\begin{lemma}
  $L = f^*\L$ for a unique $\L \in \Pic(N)$. 
\end{lemma}
Restricted to the section $N$ of $f$ we have (in $H^1(\Omega^1)$:
   \begin{equation} \label{LLK}
    \L \equiv L|_N \equiv -\frac{2}{m}K_M|_N \equiv -\frac{2}{m}(K_{M/N} + K_N).
   \end{equation}
\begin{proof}
For any fiber $F$ of $M\lra N$ we have $N_{F/M}\simeq \O_F^{\oplus m-n}$ and $\O_F = K_M|_F$ by adjunction. Because of \Formel{LK} $L|_F$ is a torsion line bundle. Consider
  \begin{equation} \label{seqs}
\xymatrix{0 \ar[r] & f^*\Omega_N^1 \ar[r] & \Omega_M^1\ar[d]^{g^{-1}} \ar[r] & f^*E^{1,0} \ar[r] & 0 \\
 0 \ar[r] & f^*E^{1,0*} \otimes L^* \ar[r] & T_M \otimes L^* \ar[r] & f^*T_N \otimes L^* \ar[r] & 0.}
   \end{equation}
Restrict to $F$. The top sequence and $f^*\Omega_N^1|_F \simeq \O_F^n$ shows $H^0(F, \Omega_M^1|_F) \not= 0$. The lower sequence shows $T_M \otimes L^*|_F$ is an extension of $L^{\oplus m-n}$ by $L^{\oplus n}$. Using $H^0(F, T_M|_F \otimes L) \simeq H^0(F, \Omega_M^1|_F) \not= 0$ we find $H^0(F, L|_F) \not= 0$.

A torsion line bundle with a section is trivial, i.e., $L|_F \simeq \O_F$. Then $\L := f_* L \in \Pic(N)$ and $L = f^*\L$. Uniqueness follows from the injectivity of $\Pic(N) \lra \Pic(M)$ ($M\lra N$ has a section). Restricted to $N$ we have $L|_N \simeq \L$.
\end{proof}

\subsection{Slope and stability} These notions depend on the choice of a polarization. We will only study bundles on $N$ and will choose $K_N$ as the polarization.

For any torsion free sheaf $\F$ on $N$ of positive rank, the slope is 
  \[\mu(F) = \frac{\det \F^{**}.K_N^{n-1}}{rk \F}.\]
A bundle $E$ is called {\em (semi-) stable} if for any proper nonzero subsheaf $\F$ of $E$ one has $\mu(\F) <(=) \mu(E)$. More generally a Higgs bundle $(E, \theta)$ is called {\em (semi-) stable} if for any proper nonzero Higgs-subsheaf $\F$ one has $\mu(\F) <(=) \mu(E)$. 

For example, Yau's uniformization theorem, nicely recalled in \cite{VZ2}, says if $K_N$ is ample then $\Omega_N^1$ is polystable, i.e., a direct sum of stable vector bundles of the same slope.

\begin{lemma} \label{zero}
  $\Hom(f^*\Omega_N^1, f^*T_N\otimes L^*) \simeq \Hom(\Omega_N^1, T_N \otimes \L^*) = 0$.
\end{lemma}

\begin{proof} The first isomorphism is the projection formula. The bundle $\Omega_N^1$ is semistable. By \cite{KBook}, 7.11., any map $\Omega_N^1 \lra T_N \otimes \L^*$ vanishes identically as soon as $\mu(\Omega_N^1) > \mu(T_N \otimes \L^*)$. 
  \[\mu(T_N \otimes \L^*) = \frac{(-K_N -n\L).K_N^{n-1}}{n} = -\mu(\Omega_N^1)-\L.K_N^{n-1}\]
and \Formel{LLK} gives $\mu(T_N \otimes \L^*) = -\mu(\Omega_N^1)+\frac{2n}{m}\mu(\Omega^1_N) +\frac{2(m-n)}{m}\mu(E^{1,0})$. We therefore have to show 
$\mu(\Omega_N^1) - \mu(T_N \otimes \L^*) = \frac{2(m-n)}{m}(\mu(\Omega_N^1) - \mu(E^{1,0}))>0$. Because of \Formel{dimassum} we have to show $\mu(\Omega_N^1) - \mu(E^{1,0})>0$.

In the case $\mu(E^{1,0}) =0$ this follows from $\mu(\Omega_N^1) > 0$ ($n > 0$). In the case $\mu(E^{1,0}) > 0$ this follows from $2\mu(E^{1,0}) \le \mu(\Omega_N^1)$ (\cite{VZ2}, Theorem~1 and Remark~2). 
\end{proof}

\subsection{Holomorphic projective connections}
The manifold $N_n$ is said to carry a {\em holomorphic (normal) projective connection}, if the normalized Atiyah class of $\Omega_N^1$ has the form
 \[a(\Omega_N^1) = \frac{c_1(K_N)}{n+1} \otimes id + id \otimes \frac{c_1(K_N)}{n+1}\]
 in $H^1(N, \Omega_N^1 \otimes T_N \otimes \Omega_N^1)$, where tensors are identified as before. The following result in the K\"ahler--Einstein case is due to Kobayashi and Ochiai (\cite{KO}):
\begin{theorem} \label{KEPC}
  Let $N_n$ be K\"ahler--Einstein. Then $N$ carries a holomorphic projective connection if and only if $N$ is
  \begin{enumerate}
   \item $\simeq \PN_n(\KC)$ or
   \item a finite {\'e}tale quotient of a torus or
   \item a ball quotient.
  \end{enumerate}
\end{theorem}
Here a ball quotient means the universal covering space of $N$ can be identified with $\BN_n(\KC) = \{z \in \KC^{n+1} \mid |z|<1\}$, the non compact dual of $\PN_n(\KC)$ in the sense of hermitian symmetric spaces. Then $\pi_1(N)$ acts like a subgroup of $SU(1,n) \simeq \HAut(\BN_n(\KC))$.

\begin{lemma} \label{ballq}
 $N$ is a ball quotient and $K_N \equiv -\frac{n+1}{2}\L$.
\end{lemma}

\begin{proof}
Again identify $N$ and $s(N)$, where $s$ denotes the section of $M \lra N$. Then
   \begin{equation} \label{split}
     \xymatrix{0 \ar[r] & \Omega_N^1 \ar[r]_{df|_N} & \Omega_M^1|_N \ar@/_1pc/[l]_{ds} \ar[r] & E^{1,0}|_N \ar[r] & 0}
   \end{equation}
splits holomorphically. The Atiyah class of a direct sum is the sum of the Atiyah classes. We may therefore compute $a(\Omega_N^1)$ by restricting $a(\Omega_M^1)$ to $N=s(N)$ and to the direct summand $\Omega_N^1$. 

The Atiyah class of $\Omega_M^1$ is given by \Prop{AtCl}. Projection of the first two summands simply yields
 \begin{equation}\label{prcN}
    \frac{c_1(K_M|_N)}{m} \otimes id + id \otimes \frac{c_1(K_M|_N)}{m}
  \end{equation}
in $H^1(N, \Omega_N^1 \otimes T_N \otimes \Omega_N^1)$. We claim that the third summand $g \otimes g^{-1}(\frac{c_1(K_M)}{m})$ vanishes under projection. The third summand goes to the tensor product of $g$ under 
 $H^0(M, S^2\Omega_M^1 \otimes L) \lra H^0(N, S^2\Omega_N^1 \otimes \L)$ and the image of $g^{-1}(\frac{c_1(K_M)}{m}) \in H^1(M, T_M \otimes L^*)$ under $H^1(M, T_M \otimes L^*) \lra H^1(N, T_N\otimes \L^*)$. We claim that this last factor vanishes.

We have $K_M = f^*H$ for some $H \in Pic(N)$. The image of $g^{-1}(\frac{c_1(K_M)}{m})$ under $H^1(M, T_M \otimes L^*) \lra H^1(N, T_N\otimes \L^*)$ is the image of $c_1(H)$ under $H^1$ of
  \[\xymatrix{\Omega_N^1 \ar[r]^{df|_N} & \Omega_M^1|_N \ar[r]^{(g|_N)^{-1}} & T_M\otimes L^*|_N \ar[r]^{ds^t} & T_N \otimes \L^*.}\]
By \Lem{zero}, this map vanishes identically. Then the third summand indeed vanishes identically. We find that $a(\Omega_N^1)$ is given by \Formel{prcN}. The trace gives
  \[K_N \equiv \frac{n+1}{m}K_M|_N \quad \mbox{or} \quad  \frac{K_N}{n+1} \equiv \frac{K_M|_N}{m}.\]
Replace this in \Formel{prcN} to see that $N$ has a projective connection. Then $K_N$ ample implies $N$ is a ball quotient by \Theo{KEPC}. Finally 
$\frac{K_N}{n+1} \equiv \frac{K_M|_N}{m}$ and \Formel{LLK} gives $K_N \equiv -\frac{n+1}{2}\L$.
\end{proof}

\Lem{zero} implies that $g^{-1}$ induces a commuting diagram
 \begin{equation} \label{seqs3}
\xymatrix{0 \ar[r] & f^*\Omega_N^1 \ar[r]\ar[d]^{\varphi} & \Omega_M^1\ar[d]^{g^{-1}} \ar[r] & f^*E^{1,0} \ar[r] \ar[d]^{\psi}& 0 \\
 0 \ar[r] & f^*E^{1,0*} \otimes L^* \ar[r] & T_M \otimes L^* \ar[r] & f^*T_N \otimes L^* \ar[r] & 0.}
   \end{equation}
Here $\varphi$ is a bundle injection while $\psi$ surjective. By the projection formula we have $\Hom(f^*\Omega_N^1, f^*(E^{1,0*}\otimes \L^*)) \simeq \Hom(\Omega_N^1, E^{1,0*}\otimes \L^*)$ implying that $\varphi$ is the pullback of some bundle injection $\Omega_N^1 \lra E^{1,0*}\otimes \L^*$. Denote the cocernel by $R$. Then
  \begin{equation} \label{seqs2} 
    0 \lra \Omega_N^1 \lra E^{1,0 *}\otimes \L^* \lra R \lra 0.
  \end{equation}

In fact $R = 0$:

\begin{lemma} \label{halb}
  $E^{1,0} \simeq T_N \otimes \L^*$ and $m = 2n$.
\end{lemma}

\begin{proof}
  Assume $R \not= 0$ or, equivalently, $\rk R = m-2n > 0$. By \Formel{seqs2} and \Formel{LLK} 
  \[\det R \equiv  -K_{M/N}|_N - (m-n)\L -K_N \equiv \frac{m- 2n}{m}(K_{M/N}|_N + K_N).\]
Then $\mu(R) = \frac{\det R.K_N^{n-1}}{m-2n} =  \frac{1}{m}(K_{M/N}|_N + K_N).K_N^{n-1} =  \frac{1}{m}((m-n)\mu(E^{1,0}) + n\mu(\Omega_N^1))$. As $\mu(E^{1,0}) \ge 0$ and $\mu(\Omega_N^1) - \mu(E^{1,0}) > 0$ (see the proof of \Lem{zero}) we find $\mu(R) > 0$.

Consider $f_*$ of \Formel{seqs3}. Let $\K = ker\{\theta^{1,0}: E^{1,0} \lra E^{0,1}\otimes \Omega_N^1\}$ and  $\K' = ker\{T_N \lra E^{1,0 *} \otimes E^{0,1}\}$. We obtain 
\[\xymatrix{& 0\ar[d] & & &\\
0 \ar[r] & \Omega_N^1 \ar[r]\ar[d] & f_*\Omega_M^1 \ar[d]^{\simeq}\ar[r] & \K \ar[r] & 0\\
0 \ar[r] & E^{1,0*} \otimes \L^* \ar[r]\ar[d] & f_*T_M\otimes \L^* \ar[r] & \K'\otimes \L^* \ar[r] & 0 \\
& R \ar[d] & & & \\
& 0 & & &.}\]
We have an induced surjective map $\K \lra \K'\otimes \L$. By the snake Lemma the kernel is isomorphic to $R$. This gives an inclusion of Higgs bundles
  \[(R, 0) \subset (\K, 0) \subset (E^{1,0} \oplus E^{0,1}, \theta).\]
The latter is semistable as a Higgs bundle. Then $\mu(R) \le \mu(E^{1,0}\oplus E^{0,1}) = 0$, a contradiction.
\end{proof}

\begin{proof} [Proof of \Prop{main}.]
 Let $f:M_m \to N_n$ be as in \Prop{main}. By \Lem{halb}, we have $m = 2n$ and $E^{1,0} \simeq T_N \otimes \L^*$ where $\L \in \Pic(N)$ and $f^*\L = L$. By \Lem{zero} and \Lem{halb}, the conformal structure $g$ is compatible with
\[0 \lra f^*\Omega_N^1 \lra \Omega_M^1 \lra f^*(T_N \otimes \L) \lra 0.\]
By \Lem{ballq}, $N$ is a ball quotient. Then $\Omega_N^1$ is stable with respect to $K_N$ and $\Hom(\Omega_N^1, \Omega_N^1)\simeq \KC$. By \Lem{CSeq}, the splitting class $e \in  H^1(M, f^*(\Omega_N^1 \otimes \Omega_N^1 \otimes \L))$ of the above sequence is alternating on $f^*T_N$. 

The beginning of Leray spectral gives a map
    \[H^1(M, f^*(\Omega_N^1 \otimes E^{1,0*})) \lra 
  H^0(N, R^1f_*f^*(\Omega_N^1 \otimes E^{1,0*}))\simeq 
   H^0(N, \Omega_N^1 \otimes E^{1,0*} \otimes E^{1,0*}).\]
It sends $e$ to the Kodaira Spencer map $\delta$. It is well known, or follows from the construction of the moduli space ${\mathcal A}_{m-n}$ of abelian varieties as the dual of the Lagrangian Grassmanian, respectively, that $\delta$ is symmetric on $E^{1,0}$.

Because of \Lem{halb} we have
  \[\delta \in H^0(N, \Omega_N^1 \otimes E^{1,0*} \otimes E^{1,0*})
\simeq H^0(N, \Omega_N^1 \otimes \Omega_N^1 \otimes \Omega_N^1 \otimes \L^{\otimes 2}).\]
We conclude that $\delta$ as a map $T_n^{3}\lra \L^{\otimes 2}$ is alternating on the first two tensors and symmetric on the latter two. An easy exercise shows $\delta = 0$.

After an \'etale base change we may assume that $f$ is induced by a map $N \lra {\mathcal A}_{m-n}$ which must now be constant. Then $\det E^{1,0}\equiv 0$ and by \Lem{halb}, $K_N \equiv -n\L$. Combined with $K_N \equiv -\frac{n+1}{2}\L$ from \Lem{ballq} we find $(\frac{1}{n}-\frac{2}{n+1})K_N \equiv 0$ and $K_N$ implies 
  $\frac{1}{n}-\frac{2}{n+1}=0$. Then $n=1$.
  
By \Lem{halb}, $M$ is a projective surface. It comes with an elliptic fiber bundle $M \lra N$ with typical fiber $\simeq F$ an elliptic curve. The map has a section and the tangent sequence \Formel{seqs} splits holomorphically because $e$ is alternating and must vanish. Then $M$ is induced by a representation $\pi_1(N) \lra Aut(F)$.
\end{proof}

\end{document}